\documentclass[12pt]{amsart}
\usepackage[utf8]{inputenc}
\usepackage[english]{babel}
\usepackage{tikz-cd}
\usepackage{verbatim}
\usepackage[margin=1in,top=0.95in, bottom=0.95in, footskip=0.25in]{geometry}
\usepackage{lipsum}
\usepackage{amsmath,amsthm,amssymb,mathrsfs}
\usepackage{thm-restate}
\definecolor{vividburgundy}{rgb}{0.62, 0.11, 0.21}
\usepackage[driverfallback=hypertex,pagebackref=false,colorlinks,citecolor=vividburgundy]{hyperref}
\usepackage[capitalize]{cleveref}
\usepackage{tikz-cd}
\usepackage{tzplot}
\tikzset{down/.style={anchor=south, rotate=-45, inner sep = -.5mm}}
\tikzset{up/.style={anchor=south, rotate=45, inner sep=-.5mm}}

\newtheorem{theorem}{Theorem}[section]

\newtheorem{proposition}[theorem]{Proposition}

\newtheorem{problem}[theorem]{Problem}

\newtheorem{lemma}[theorem]{Lemma}
\newtheorem{conjecture}[theorem]{Conjecture}

\theoremstyle{definition}
\newtheorem{definition}[theorem]{Definition}
\newenvironment{prf}{\noindent{\bf Proof.~}}{\(\qed\)}
\newcommand{\BPF}{\begin{prf}} 
\newcommand {\EPF}{\end{prf}}

\newcommand{\Q}{\mathbb{Q}}
\newcommand{\Z}{\mathbb{Z}}
\newcommand{\R}{\mathbb{R}}
\newcommand{\C}{\mathbb{C}}
\newcommand{\rK}{\mathcal{K}}
\newcommand{\rL}{\mathcal{L}}
\newcommand{\rM}{\mathcal{M}}

\newcommand{\rO}{\mathcal{O}}

\newcommand{\Ind}{\text{Ind}}

\title[Brauer-Siegel theorem for families of number fields over almost $S_n$ fields]{Brauer-Siegel theorem for families of number fields over almost $S_n$ fields}

\author{Anup B. Dixit}

\address{Department of Mathematics\\ Institute of Mathematical Sciences (HBNI)\\ CIT Campus, IV Cross Road\\ Chennai\\ India-600113}
\email{anupdixit@imsc.res.in}



\begin{document}

\begin{abstract}
    The classical Brauer-Siegel conjecture describes the asymptotic behavior of the product of the class number and the regulator in families of number fields. All known cases of the conjecture rely on reducing the problem, via group-theoretic methods, to Siegel's theorem for quadratic fields over $\mathbb{Q}$ or over a fixed base field. In this paper, we establish a new form of descent for the Brauer–Siegel conjecture. We show that if the conjecture holds for a family of almost $S_n$-fields, it necessarily holds for all quadratic extensions over that family, under mild conditions. This result may be viewed as an analogue of Siegel's theorem in which the base field is allowed to vary. In addition, we also establish the generalized Brauer-Siegel conjecture as formulated by Tsfasman-Vl\u{a}du\c{t} for asymptotically good towers of number fields over a family of almost $S_{n}$-fields.

\end{abstract}

\maketitle

\section{Introduction}
\bigskip

Let $K/\Q$ be a number field. Denote by $n_K=[K:\Q]$ its degree, by $d_K:= |disc(K/\Q)|$ its absolute discriminant, by $h_K$ its class number, by $R_K$ its regulator and by $\widetilde{K}$ its Galois closure over $\Q$. A classical problem in number theory, dating back to Gauss, is to understand how $h_K$ varies on varying $K$. If one uses analytic methods, it is more feasible to study the distribution of $h_K R_K$, which appears in the class number formula. Writing $\rho_K$ for the residue of the Dedekind zeta-function $\zeta_K(s)$ at $s=1$, the class number formula states that
\begin{equation*}
    \rho_K = \frac{2^{r_1}(2\pi)^{r_2} h_K R_K}{\mu_K\, \log \sqrt{d_K}},
\end{equation*}
where $r_1$ and $r_2$ are the number of real and complex embeddings of $K$ respectively and $\mu_K$ denote the number of roots of unity in $K$. Consequently, the asymptotic behaviour of $h_KR_K$ in a family is governed by that of $\rho_K$, which is amenable to analytic techniques. In the case of a family of quadratic fields $K_i/\Q$, Siegel \cite{Sie} proved that
\begin{equation}\label{BS-Siegel}
    \lim_{i\to\infty} \frac{\log h_{K_i}R_{K_i}}{\log \sqrt{d_{K_i}}} = 1.
\end{equation}
Thus,  $\log h_K R_K$ grows asymptotically like $\log \sqrt{d_K}$ over a family of quadratic fields. In particular, there are only finitely many quadratic fields with bounded $h_KR_K$, even though it is conjectured that $h_K=1$ for infinitely many real quadratic fields. This is famously known as the class number one problem. In the rest of the article, following the notation in \cite{TV}, we denote $g_K:=\log \sqrt{d_K}$.\\

Siegel's theorem was extended by Brauer in \cite{BS}, who showed that for any family of number fields $\{K_i\}$ with $K_i/\Q$ Galois for all $i$, if $d_{K_i}^{1/n_{K_i}} \to\infty$, then \eqref{BS-Siegel} holds. This was further extended to families of almost normal fields by Stark \cite{Stk} and Zykin \cite{Zyk}. Here, we call a number field $L$ almost normal if there is a tower
\begin{equation*}
    \Q=L_1\subsetneq L_2\subsetneq\cdots \subsetneq L_t=L,
\end{equation*}
such that $L_{i+1}/L_i$ is Galois for all $i$. For families of number fields $K$ with solvable Galois closure, \eqref{BS-Siegel} was proved by the author in \cite{Dixit} building on the work of V. K. Murty \cite{Km}. More recently Wong \cite{peng} proved \eqref{BS-Siegel} for suitable combinations of the above conditions.\\

In all the above cases, a crucial input is the existence of a zero-free region for $\zeta_K(s)$ in a neighbourhood of $s=1$. More precisely, Stark \cite{Stk} showed that if $K/\Q$ is a number field for which $\zeta_K(s)$ has a real zero $\beta$ satisfying
\begin{equation*}
    1 - \frac{1}{4 f(n_K) \log d_K} \leq \beta < 1,
\end{equation*}
then there exists a quadratic subfield $N\subset K$ such that $\zeta_N(\beta)=0$. This reduction to quadratic case is the key mechanism underlying existing proofs. Stark further showed that one may take $f(n_K)=1$ when $K/\Q$ is Galois, and $f(n_K)=4$ when $K/\Q$ is almost normal. For number fields whose Galois closure over $\Q$ is solvable,  V. K. Murty \cite{Km} showed that one can take 
\begin{equation*}
    f(n_K) = c n_K^{e(n_K)} \delta(n_K),
\end{equation*} 
where $e(n_K) = \max_{p^a \| n_K} a$ and $\delta(n_K) \ll n_K^4$ is an explicitly described function. These estimates are sufficiently strong to permit a reduction of such families to that of quadratic fields, where Siegel's theorem is applicable.
\medskip

For a general extension $K/\Q$, however, the only unconditional bound due to Stark \cite{Stk} allows one to take $f(n_K)= n_K!$, and grows too rapidly to yield \eqref{BS-Siegel}. However, under the Generalized Riemann Hypothesis,  $\zeta_K(s)$ has no zeros in the region $1/2<\Re(s)<1$ and hence we expect $\eqref{BS-Siegel}$ to hold in general. This leads to the Brauer-Siegel conjecture, which we refer to as \textit{(BS)} and state below.


\begin{conjecture}[BS]\label{BS-conj}
    Let $\{K_i\}$ be a family of number fields satisfying $d_{K_i}^{1/n_{K_i}} \to \infty$. Then,
        \begin{equation*}
            \lim_{i\to\infty} \frac{\log h_{K_i}R_{K_i}}{g_{K_i}} = 1.
        \end{equation*}
        Applying the class number formula, this is equivalent to 
        \begin{equation*}
            \lim_{i\to\infty} \frac{\log \rho_{K_i}}{ g_{K_i}} = 0,
        \end{equation*}
        where $\rho_{K_i}$ denotes the residue of $\zeta_K(s)$ at $s=1$.
\end{conjecture}


Recent years have seen substantial interest in various aspects of this conjecture, such as its effective versions, see for instance \cite{asif}. In this work, we focus on establishing a reduction-type result in this setting. More precisely, in all the cases where Conjecture \ref{BS-conj} is currently known, Siegel's theorem for quadratic fields plays an indispensable role. This leads to the question of whether the Brauer-Siegel property is preserved under quadratic extensions of number fields.


\begin{problem}
    Let $\{L_i\}$ be a family of number fields for which (BS) holds. For any $\{N_i\}$ with $[N_i:L_i]=2$, does (BS) hold for $\{N_i\}$?
\end{problem}

This problem may be viewed as a natural extension of Siegel's theorem, in which the base field is no longer fixed. It is also closely related to Stark's conjecture asserting that there are only finitely many $CM$-fields with bounded class number. In this paper, we answer this question in the affirmative when $L_i$'s are \textit{almost} $S_n$ number fields.

\medskip

\begin{definition}
    Let $L/\Q$ be a number field such that there exists an $S_n$-field $K/\Q$, satisfying $K\subset L\subset \widetilde{K}$ and $Gal(\widetilde{K}/L) \cong  S_m$ for some $m\leq n-1$. Then, we say that $L$ is an almost $S_n$-field.
\end{definition}


\begin{figure}
\begin{tikzpicture}[scale=1.1]

\node (Q)  at (0,0) {$\mathbb{Q}$};
\node (K)  at (0,2) {$K$};
\node (L)  at (0,4) {$L$};
\node (Kt) at (4,5) {$\widetilde{K}$};

\draw (Q) -- (K) -- (L);

\draw (Kt) -- node[midway, sloped, above] {$S_n$} (Q);
\draw (Kt) -- node[midway, sloped, above] { $S_{n-1}$} (K);
\draw (Kt) -- node[midway, sloped, above] {$S_m$} (L);

\end{tikzpicture}

\caption{}
\end{figure}

For instance all $S_n$-fields are almost $S_n$-fields with $m=n-1$ and this notion captures a much larger family of number fields.

\begin{theorem}\label{BS-almost-S_n}
   Let $\{L_i\}$ be a family of almost $S_{n_i}$-fields. Let $N_i/L_i$ be quadratic extensions satisfying $N_i\cap \widetilde{L_i} = L_i$ and 
   \begin{equation}\label{condition}
        \frac{\log d_{N_i}}{\log d_{L_i}} \to \infty.
   \end{equation}
    If Conjecture \ref{BS-conj} (BS) holds for the family $\{L_i\}$, then it also holds for the family $\{N_i\}$.
\end{theorem}

{\bf Remark.} One may be tempted to extend the above result to families of almost normal fields over almost $S_n$-fields. Indeed, by Stark's result (see Theorem \ref{Stark-almost-normal}), it is possible to reduce the problem to a family of quadratic extensions over $\{L_i\}$. Unfortunately, condition \eqref{condition} is an obstruction here and prevents this generalization. However, this is not an issue in the modern formulation of the Brauer-Siegel theorem for asymptotically good families introduced by Tsfasman and Vl\u{a}du\c{t} in \cite{TV}. We discuss this below.
\medskip

For a number field $K/\Q$ and a prime power $q$, denote by $N_q(K)$ the number of ideals in $\rO_K$ with norm $q$. For a family of number fields $\rK=\{K_i\}$, we say that $\rK$ is asymptotically exact if the following limits hold.
\begin{equation*}
    \phi_{\R}(\rK) = \lim_{i\to\infty} \frac{r_1(K_i)}{g_{K_i}}, \, 
    \phi_{\C}(\rK) = \lim_{i\to\infty} \frac{r_2(K_i)}{g_{K_i}}\,\, \, \text{and}\,\,\,\,
    \phi_{q}(\rK) = \lim_{i\to\infty} \frac{N_q(K_i)}{g_{K_i}},
\end{equation*}
for all prime powers $q$. The generalized Brauer-Siegel conjecture, referred to as GBS, as formulated by Tsfasman and Vl\u{a}du\c{t} \cite{TV} is as follows. 
\begin{conjecture}[GBS]\label{BS}
For any asymptotically exact family $\mathcal{K}=\{K_i\}$,
\begin{equation*}\label{BS1}
\lim_{i\to\infty} \frac{\log h_{K_i}R_{K_i}}{g_{K_i}}=1 + \sum_q \phi_q \log\frac{q}{q-1} -\phi_{\mathbb{R}} \log 2 - \phi_{\mathbb{C}}\log 2\pi.
\end{equation*}
Using the class number formula, the above statement is equivalent to
\begin{equation*}\label{BS2}
\lim_{i\to\infty} \frac{\log \rho_{K_i}}{g_{K_i}} =  \sum_q \phi_q \log\frac{q}{q-1}.
\end{equation*}
\end{conjecture}
For an asymptotically exact family, if $\lim_i \frac{n_{K_i}}{g_{K_i}}>0$, we call it to be \textit{asymptotically good}. For instance, an infinite Hilbert class field tower $\{K_i\}$ over $K$, where the root discriminant remains a constant is an example of an asymptotically good family. Note that $\lim_i \frac{n_{K_i}}{g_{K_i}}=0$ is equivalent to the classical case that the root discriminant $d_{K_i}^{1/n_{K_i}} \to \infty$. For asymptotically good families, we obtain the following. 

\begin{theorem}\label{asymp-good-theorem}
    Let $\rL=\{L_i\}$ be an asymptotically good family satisfying GBS. For each $i$, suppose $L_i$ is an almost $S_{n_i}$-field, but not an $S_{n_i}$-field.  Let $\rM=\{M_i\}$ be an asymptotically good tower of number fields satisfying $M_i\cap \widetilde{L_i}=L_i$. Then, GBS holds for $\rM$ if for all $i$, either
    \begin{enumerate}
        \item $M_i/L_i$ is almost normal \,\,\,\, or 
        \item $M_i/L_i$ has solvable Galois closure. 
    \end{enumerate}
\end{theorem}  

The limitation of the above result is that it does not include family of $S_n$-fields, i.e., we force $m<n-1$. On the other hand, unlike the case of Theorem \ref{BS-almost-S_n}, we no longer require condition \eqref{condition} and hence have the result in more generality.




\bigskip
\section{Notation and Preliminaries}
\bigskip

In this section, we recall some facts and lemmata towards the proof of our main theorem. We first introduce some basic group theoretic constructions. Our method is inspired from that in Hoffstein and Jochnowitz \cite{Hoff}, where they establish a zero free region for $\zeta_K/\zeta_k(s)$, where  $K$ is a CM-field over a totally real $S_n$-field $k$. We largely maintain the notation in their paper.
\medskip

 Let $S_m$ denote the symmetric group of $m$ elements. It is well known that the irreducible complex characters of $S_m$ are in bijection with the partitions of $m$. Let $\chi_1,\chi_2,\cdots, \chi_m$ denote the irreducible characters of $S_m$ corresponding to the partitions 
\begin{equation*}
    [m], [m-1,1],[m-2,1,1],\ldots, [1,1,\ldots,1]
\end{equation*}
 respectively. These partitions are commonly referred to as hooks. The character $\chi_1$ is the trivial character, while $\chi_m$ is the sign character, which is the unique quadratic non-trivial character of $S_m$, for $m\geq 5$.
 \medskip

For $0\leq i\leq m$, let $\psi_i$ denote the character of the subgroup $S_{m-i}\times S_i \subset S_m$, that is trivial on the factor $S_{m-i}$ and equal to the alternating character on $S_i$. Let $\psi_i^*$ denote the character of $S_m$ induced from $\psi_i$. By an application of Littlewood-Richardson rule, one obtains the decomposition
\begin{equation}\label{psi_i-decomposition}
    \psi_i^* = \chi_i \oplus \chi_{i+1}.
\end{equation}

Next, let $\theta$ denote the unique non-trivial character of the cyclic group $\Z/2\Z$. Define the character $\psi_i' := \psi_i\otimes \theta$ on $S_m\times \Z/2\Z$. Similarly, define $\chi_i':=\chi_i\otimes \theta$ on $S_m\times \Z/2\Z$. Then, from \eqref{psi_i-decomposition}, one can deduce that
\begin{equation*}
    (\psi_i')^* = \chi_i'  \oplus \chi_{i+1}'.
\end{equation*}
\bigskip

We also recall some known results on the exceptional zeros of Dedekind zeta-functions. For a number field $K/\Q$, Stark \cite{Stk} showed that $\zeta_K(s)$ has at most one zero in the region
\begin{equation}\label{stark-region}
    1-\frac{1}{4\log d_K} \leq \Re(s) <1\,\,\, \text{and}\,\,\, |\Im(s)|\leq \frac{1}{4\log d_K}.
\end{equation}
If such a zero exists, it is necessarily real and simple. Here and in the rest of the paper, $\log$ is chosen with the principal branch. \\ 

We next recall two famous results of Stark \cite{Stk} and V. K. Murty \cite{Km}, which describe how the existence of such zeros in a narrow neighbourhood of $s=1$ can be reduced to a quadratic subextension.

\begin{theorem}\label{Stark-almost-normal}(Stark)
    Let $L/K$ be an almost normal extension, i.e.,
    \begin{equation*}
        K=L_0\subsetneq L_1\subsetneq L_2\subsetneq \ldots\subsetneq L_t=L
    \end{equation*}
    such that for all $1\leq i\leq t$, $L_i$ is Galois over $L_{i-1}$. Suppose $\zeta_L(\beta)=0$ for a real $\beta$ in the range
    \begin{equation*}
        1-\frac{1}{16 \log d_L} \leq \beta<1.
    \end{equation*}
    Then, there exists a field $N$ with $K\subseteq N\subseteq L$, such that $\zeta_N(\beta) = 0$ and $[N:K]\leq 2$.
    \end{theorem}

\begin{theorem}\label{Km}(V. K. Murty)
Let $L/K$ be an extension of degree $n$ whose Galois closure is solvable. Let
\begin{align*}
e(n) & := \max_{p^\alpha || n} \alpha, \\
\delta(n) & := (e(n) + 1)^2 \, \, 3^{1/3} \, \, 12^{(e(n) -1)}.
\end{align*}
Then, there exists an absolute constant $c>0$, such that if $\zeta_L(s)$ has a real zero in the region
\begin{equation}\label{region}
1 - \frac{c}{n^{e(n)} \delta(n) \log d_L} \leq \beta < 1,
\end{equation}
then there exists a field $N$ with $K\subseteq N\subseteq L$, such that $\zeta_N(\beta) = 0$ and $[N:K]\leq 2$.
\end{theorem}
\medskip

An important role in our proof is played by Stark's lemma, which relates the zero-free region of $\zeta_K(s)$ near $s=1$ to a lower bound on $\rho_K$. Let $\beta_0(K)$ denote the possible exceptional zero of $\zeta_K(s)$ in the region $\beta_0(K) \in [1-(4\log d_K)^{-1}, 1]$. Define 
\begin{equation}\label{beta_0-beta}
    \beta(K):= \max\left(\beta_0(K), 1-\frac{1}{4\log d_K}\right).
\end{equation}

\begin{lemma}\label{Stark's lemma}
    Let $\rho_K$ denote the residue of $\zeta_K(s)$ at $s=1$.  Then, there exists an absolute constant $c>0$ such that
    \begin{equation*}
        \rho_K > c\, (1-\beta(K)).
    \end{equation*}
\end{lemma}

We also recall a lemma of Stark \cite[Lemma 5]{Stk} on lower bound of $L(1,\chi)$ for a quadratic character $\chi$.
\begin{lemma}\label{Stark-L(1,chi)}
    Let $N/L$ be a quadratic extension. Write $\zeta_N(s) = \zeta_L(s) L(s,\chi)$. Then, $L(s,\chi)$ has at most one zero $\beta_0$ in the interval $[1-(4\log d_N)^{-1}, 1]$. Let $\beta_{\chi} = \max(1-(4\log d_N)^{-1}, \beta_0)$. For any $\sigma_1 \in [1+ (4\log d_N)^{-1}, 2]$, we have
    \begin{equation*}
        L(1,\chi) > c \left(\frac{1-\beta_{\chi}}{\sigma_1 -1}\right) d_L^{-\frac{\sigma-1}{2}} \left(\frac{\sqrt{\pi}}{\Gamma(\sigma_1/2)}\right)^{r_1(L)}\left(\frac{2^{\sigma_1-1}}{\Gamma(\sigma_1}\right)^{r_2(L)} \left(\frac{\pi^{\frac{\sigma_1-1}{2}}}{\zeta(\sigma_1}\right)^{n_L},
    \end{equation*}
    where $c$ is an absolute constant, $r_1(L)$ and $r_2(L)$ are the number of real and complex embeddings of $L$ respectively.
    
\end{lemma}

We now show that distinct conjugate quadratic fields gives rise to a biquadratic field whose zeta-function has zeros of higher multiplicity. This is the main idea of Siegel's proof for quadratic fields.
\begin{lemma}\label{biquad-lemma}
    Let $K$ be a number field and let $N_1$ and $N_2$ be quadratic extensions of $K$ that are not contained in $\widetilde{K}$. Suppose that $N_1$ and $N_2$ are Galois conjugates over $\Q$. Then, any zero $\beta$ of $\zeta_{N_i}(s)/\zeta_K(s)$ in the region
    \begin{equation*}
        \beta \in \bigg[1-\frac{1}{4\log d_{N_i}}, 1\bigg]
    \end{equation*}
    is a zero of $\zeta_{N_1 N_2}(s)$ with multiplicity at least $2$.
\end{lemma}
\begin{proof}
    By assumption, the compositum $N_1 N_2$ is a Galois extension of $K$ with Galois group $(\Z/2\Z)^2$. Let $\chi_1$ and $\chi_2$ denote the quadratic characters of $Gal(\overline{K}/K)$ corresponding to the extension $N_1/K$ and $N_2/K$. Then, we may write
    \begin{equation*}
        \zeta_{N_1}(s) = \zeta_K(s) L(s,\chi_1)\,\,\,\text{and}\,\,\, \zeta_{N_2}(s) = \zeta_K(s) L(s,\chi_2).
    \end{equation*}
    By Artin decomposition, we have
    \begin{equation*}
        \zeta_{N_1 N_2}(s) = \zeta_K(s) L(s,\chi_1) L(s,\chi_2) L(s,\chi_{12}),
    \end{equation*}
    where $\chi_{12}$ is the quadratic character corresponding to the third quadratic subextension of $N_1N_2/K$. Since $N_1$ and $N_2$ are conjugates, $\zeta_{N_1}(s) = \zeta_{N_2}(s)$ and hence $L(s,\chi_1) = L(s,\chi_2)$. Moreover, by \eqref{stark-region}, each $\zeta_{N_i}(s)$ has at most one zero in the region $[1-(4\log d_{N_i})^{-1}, 1]$. If such a zero arises from $\zeta_K(s)$, then the quotient $\zeta_{N_i}/\zeta_K(s)$ has no zeros in this region. On the other hand, if the zero arises from $L(s,\chi_i)$, it becomes a zero of multiplicity at least $2$ of $\zeta_{N_1N_2}(s)$.
\end{proof}

\bigskip

\section{Brauer-Siegel theorem for quadratic extensions over almost $S_n$ fields}
\bigskip



In this section, we prove Theorem \ref{BS-almost-S_n}. Let $K$ be an $S_n$-field and $L$ be an almost $S_n$-field satisfying $K\subset L \subseteq \widetilde{K}$, with $Gal(\widetilde{K}/L) \cong S_m$.
\smallskip

Let $R^*/R$ be an extension of number fields such that $L \subseteq R \subseteq R^*\subseteq \widetilde{L}$ and $Gal(R^*/R)\cong S_r$. Let $N/L$ be a quadratic extension such that $N\cap \widetilde{K} = L$. Write $M:= NR$ and $M^* = N R^*$ quadratic extensions over $R$ and $R^*$ respectively with $Gal(M^*/R)\cong S_r \times \Z/2\Z$. (see Figure 2.)
\medskip

Let $\chi_1,\chi_2,\ldots, \chi_r$ be the characters of $S_r$ associated to the hook partitions as described in the preliminaries. Denote by $\chi_i' := \chi_i\otimes \theta$, where $\theta$ is the non-trivial character of $\Z/2\Z$. Suppose that $t:=[R:L]$.

\begin{figure}
\begin{tikzpicture}[scale=1.1]

\node (L)   at (0,0) {$L$};
\node (R)   at (2,1) {$R$};
\node (Rs)  at (4,2) {$R^*$};
\node (K) at (0,-1) {$K$};
\node (Q) at (0,-2.5) {$\Q$};
\node (Lt) at (6,3) {$\widetilde{L}=\widetilde{K}$};
\node (N) at (-0.5,1.5) {$N$};
\node (M)  at (1.5,2.5) {$M$};
\node (Ms) at (3.5,3.5) {$M^* $};
\node (Nt) at (5.5,4.5) {$\widetilde{N}$};

\draw (K) -- node[midway, right]
  {$n$} (Q);
\draw (L) -- (K);
\draw (L) -- node[midway, sloped, above] {$t$} (R);
\draw (R) -- node[midway, right] {} (Rs);


\draw (Lt) -- node[midway, sloped, above]
  {} (Rs);


\draw (N) -- node[midway, sloped, above]
  {$2$} (L);


\draw (N) -- (M);
\draw (R) -- (M);
\draw (Ms) -- (Nt);
\draw (Lt) -- (Nt);

\draw (Ms) -- node[midway, sloped, above]
  {$ S_r$} (M);
\draw (Rs) -- (Ms);

\draw (Rs) -- node[midway, sloped, above]
  {$ S_r$} (R);

\draw[bend left=25] 
  (Lt) to node[midway, sloped, below]
  {$ S_m$} (L);
  
\draw[bend left=25] 
  (Lt) to node[midway, sloped, below]
  {$ S_n$} (Q);

\end{tikzpicture}
\caption{}
\end{figure}

\begin{lemma}\label{both-zeros}
 The Artin $L$-functions $L(s,\chi_1')$ and $L(s,\chi_r')$, cannot both have zeros in the interval 
 \begin{equation*}
     \bigg[1-\frac{1}{8tm\log d_N},1\bigg].
 \end{equation*}
\end{lemma}
\begin{proof}
    Recall that $\chi_1$ and $\chi_r$ correspond to the trivial and alternating characters of $S_r$. Thus, $\chi_1'$ and $\chi_r'$ are quadratic characters of $S_r\times \Z/2\Z$. Therefore, both $L(s,\chi_1')$ and $L(s,\chi_m')$ are entire. Moreover, it is easy to see that $\zeta_{M}/\zeta_R(s) = L(s,\chi_1')$ and there exists a quadratic extension $M'$, (given by the fixed field of the kernel of the quadratic character $\chi_m'$), such that $L(s,\chi_m) = \zeta_{M'}/\zeta_R(s)$. Furthermore, $M$ and $M'$ are both in $M^*\subset\widetilde{N}$. Now, considering the biquadratic extension $MM'/R$, we note that
    \begin{equation*}
        \zeta_{MM'}(s) = \zeta_R(s) L(s,\chi_1') L(s,\chi_m') L(s,\chi),
    \end{equation*}
    where $\chi$ is also a quadratic character. Hence, $L(s,\chi)$ is entire. By \eqref{stark-region}, $\zeta_{MM'}(s)$ has at most one zero in the region $[1-(4\log d_{MM'})^{-1},1]$. Using \cite[Lemma 7]{Stk}, we have that $d_{\widetilde{N}} \mid d_N^{m.m!}$. Therefore 
    \begin{equation*}
        d_{MM'} \leq (d_{\widetilde{N}})^{1/[\widetilde{N}: MM']}\leq d_N^{m.m!/[\widetilde{N}: MM']} \leq d_N^{2tm}.
    \end{equation*} 
    Hence,  we cannot have both $L(s,\chi_1')$ and $L(s,\chi_m')$ vanishing in the interval $[1-(8tm\log d_N)^{-1},1]$.
\end{proof}

\begin{lemma}\label{intermediate-existence}
    Let $R^*/R$ and $M^*/M$ be as above. If $\zeta_{M}/\zeta_R(s)$ has a real zero $\beta\in [1-(8tm\log d_N)^{-1},1]$, then there exists $R''$ satisfying $R\subset R'' \subset R^*$ with $Gal(R^*/R'')\cong S_{r-j}\times S_j$ for some $j\geq 1$ and $\zeta_{M''}/\zeta_{R''}(\beta)=0$, where $M'' = NR''$. (see Figure 3.)
\end{lemma}

\begin{figure}
\begin{tikzpicture}[scale=1.1]

\node (L)   at (0,0) {$L$};
\node (R)   at (2,1) {$R$};
\node (Rz) at (4,2) {$R''$}; 
\node (Rs)  at (6,3) {$R^*$};
\node (Lt) at (8,4) {$\widetilde{L}=\widetilde{K}$};
\node (N) at (-0.5,1.5) {$N$};
\node (M)  at (1.5,2.5) {$M$};
\node (Mz) at (3.5, 3.5) {$M''$};
\node (Ms) at (5.5,4.5) {$M^* $};
\node (Nt) at (7.5,5.5) {$\widetilde{N}$};

\draw (L) -- node[midway, sloped, above] {$t$} (R);
\draw (R) -- node[midway, right] {} (Rz);

\draw (Lt) -- node[midway, sloped, above]
  {} (Rs);

\draw (Rs) -- node[midway, sloped, above]
  {$S_{r-j}\times S_j$} (Rz);

\draw (N) -- node[midway, sloped, above]
  {$2$} (L);

\draw (N) -- (M);
\draw (R) -- (M);
\draw (Ms) -- (Nt);
\draw (Lt) -- (Nt);

\draw (Ms) -- (Mz);
\draw (Mz) -- (M);
\draw (Rs) -- (Ms);
\draw (Rz) -- (Mz);

\draw[bend left=25] 
(Rs) to node[midway, sloped, below]
  {$ S_r$} (R);

\draw[bend left=30] 
  (Lt) to node[midway, sloped, below]
  {$ S_m$} (L);

\end{tikzpicture}
\caption{}
\end{figure}

\begin{proof}
    Let $\psi_i$ be the character of $S_{r-i}\times S_i$ which acts trivially on $S_{r-i}$ and alternatingly on $S_i$, as described in the preliminaries. Then, its induced character on $S_r$ satisfies $\psi_i^* = \chi_i + \chi_{i+1}$. Let $\psi_i'$ be the character on $S_{r-i}\times S_i\times \Z/2\Z$ defined by $\psi_i' = \psi_i\otimes \theta$, where $\theta$ is the non-trivial character on $\Z/2\Z$.\\
    
    Over the extension $R^*/R$, consider the fixed field of $S_{r-j}\times S_j$. Call it $F_j$. Since $\psi_j'$ is a quadratic character of $Gal(M^*/F_j)$, there exists a quadratic extension $N_j/F_j$ such that $\zeta_{N_j}/\zeta_{F_j}(s) = L(s, \psi_j')$. Since Artin $L$-functions are invariant under induction and $\psi_j^* = \chi_j + \chi_{j+1}$, we have
    \begin{equation*}
        L(s, \psi_j') = L(s,\chi_j') L(s, \chi_{j+1}').
    \end{equation*}
    Suppose none of the functions $L(s,\chi_1')L(s,\chi_2'), L(s,\chi_2')L(s,\chi_3'), \ldots, L(s,\chi_{m-1}')L(s,\chi_r') $ have a zero at $\beta$. Since $L(\beta, \chi_1')=0$, we must have poles and zeroes alternatively for $L(s,\chi_2')$, $L(s,\chi_3')$, $\ldots, L(s,\chi_r')$ at $s=\beta$. But, by Lemma \ref{both-zeros}, $L(s,\chi_r')$ can neither have a zero or pole at $\beta$. This leads to a contradiction. Hence, there exists a $j\geq 1$ for which $L(\beta, \psi_j')=0$.
\end{proof}

For a quadratic extension $N/L$, we now establish a zero free region for $\zeta_N/\zeta_L(s)$, which is at the heart of the proof of the main theorems. 


\begin{theorem}\label{zero-free-almost-S_n}
    Let $L$ be an almost $S_n$-field and $Gal(\widetilde{L}/L)\cong S_m$. Let $N/L$ be a quadratic extension satisfying $N\cap \widetilde{L} = L$. Then, any real zero $\beta$ of $\zeta_N/\zeta_L(s)$ satisfies
    \begin{equation*}
        1-\beta \geq \frac{1}{4^{m+2} m\log d_N}.
    \end{equation*}
\end{theorem}

\begin{proof}
Suppose $\beta $ is a real zero of $\zeta_N/\zeta_L(s)$ satisfying $\beta \in [1-(4^{m+2} m\log d_N)^{-1}, 1]$. By Lemma \ref{intermediate-existence}, there exists $L_1$ such that $L\subseteq L_1 \subseteq \widetilde{L} = \widetilde{K}$ and $Gal(\widetilde{L}/L)\cong S_{m-j}\times S_j$, satisfying $\zeta_{NL_1}/\zeta_{L_1} (\beta)=0$. Without loss of generality, assume $m-j<j$. Next, define $L_1^*$ to be the fixed field of $S_j$ in $\widetilde{L}/L$. Thus, we have $Gal(L_1^*/L_1) \cong S_{m-j}$. We apply the Lemma \ref{intermediate-existence} again and obtain that there exists $L_2$ such that $L_1\subseteq L_2 \subseteq L_1^*$ with $Gal(L_1^*/L_2)\cong S_{m-j-k}\times S_k$ and $\zeta_{NL_2}/\zeta_{L_2} (\beta)=0$. Again WLOG, $m-j-k<k$. We continue this process until we reach the group of the type $S_a\times S_1$. In the final step, we have the extension $L_g\subseteq L_{g+1}\subseteq L_g^*$ such that $Gal(L_g^*/L_{g+1}) \cong S_a \times S_1$.\\

Now, we also have $Gal(\widetilde{K}/L_{g+1})\cong H \times S_1$ where $H\subseteq S_{n-2}$. Hence, $L_{g+1}$ contains $L$ and a conjugate of $L$, say $L'$. Write $L=\Q(\alpha)$ and $L'=\Q(\alpha')$ such that $N= \Q(\sqrt{\alpha})$, where $\alpha$ and $\alpha'$ are conjugates. Call $N'=\Q(\sqrt{\alpha'})$. Then $N$ and $N'$ are conjugates. Moreover, $N_{g+1}= NL_{g+1}= L_{g+1}(\sqrt{\alpha})$ and $N_{g+1}'= NL_{g+1}'= L_{g+1}(\sqrt{\alpha'})$ are also conjugates. Applying Lemma \ref{biquad-lemma} and \eqref{stark-region}, we deduce that $\zeta_{N_{g+1}}(s)$ does not have any zero in the region
\begin{equation*}
    \bigg[1-\frac{1}{4\log d_{N_{g+1}N_{g+1}'}},1\bigg] \supset \bigg[1-\frac{1}{16\log d_{N_{g+1}}},1\bigg].
\end{equation*}
Write $t=[N_{g+1}:N]$. Applying \cite[Lemma 7]{Stk}, we have $d_{\widetilde{N}} \mid d_N^{m.m!}$ and hence $d_{N_{g+1}} \leq d_{\tilde{N}}^{1/[\widetilde{N}:N_{g+1}]} \leq d_N^{mt}$. Therefore, we have
\begin{equation}\label{eqn-1}
    1-\beta > \frac{1}{16mt \log d_N}.
\end{equation}
To get a bound on $[L_{g+1}:L]$, note that $[L_1:L]= \frac{m!}{(m-j)! j!}\leq 2^m$. Since $m-j<j$, in the second step, we have $[L_2:L_1]\leq \frac{(m-j)!}{(m-j-k)! k!} \leq 2^{m-j}\leq 2^{m/2}$. Similarly, in the third step, we have $[L_3:L_2]\leq 2^{m/4}$ and so on. Hence, 
$$
[L_{g+1}:L] \leq 2^{m+m/2+m/4+\cdots}\leq 4^m.
$$
Using this in \eqref{eqn-1}, we deduce $1-\beta > (4^{m+2}m \log d_N)^{-1}$, which contradicts our assumption.


\end{proof}
\medskip

\begin{proof}[Proof of Theorem \ref{BS-almost-S_n}]
For the family $\{N_i\}$, we would like to show that
$$
    \lim_{i\to\infty} \frac{\log \rho_{N_i}}{\log \sqrt{d_{N_i}}} = 0.
$$
It is known that (see for instance \cite[Chapter XVI, Lemma 1]{lang})
$$
    \limsup_{i\to\infty} \frac{\log \rho_{N_i}}{\log \sqrt{d_{N_i}}} \leq 0.
$$
Therefore, it suffices to prove that
$$
    \liminf_{i} \frac{\log \rho_{N_i}}{\log d_{N_i}} \geq 0.
$$
Let $N/L$ be a quadratic extension and let $L$ be an almost $S_n$-field. Suppose $\beta_0$ is the possible real zero of $\zeta_N(s)$ in the region $[1-(4\log d_N)^{-1},1]$. By Stark's Lemma \ref{Stark's lemma},
$$
    \rho_N \geq c(1-\beta),
$$
where $\beta=\max (1-(4\log d_N)^{-1}, \beta_0)$. Suppose $\beta_0$ is not a zero of $\zeta_L(s)$. Then, by Theorem \ref{zero-free-almost-S_n}, we deduce
$$
    \rho_N \geq c\, 4^{m} m \log d_N.
$$
In this case,
$$
    \frac{\log \rho_N}{\log d_N} \gg \frac{m}{\log d_N} + \frac{\log\log d_N}{\log d_N}.
$$
Since $d_{N_i}^{1/n_{N_i}} \to\infty$ and $m\leq n_{L_i}$, the RHS above tends to $0$ in the family and we have that (BS) holds. Now, suppose $\zeta_L(\beta_0)=0$. Then, we can write
\begin{equation*}
    \rho_N = \rho_L \, L(1,\chi),
\end{equation*}
where $L(s,\chi) = \zeta_N/\zeta_L(s)$. Since $\beta_0$ is not a zero of $L(s,\chi)$ and $\zeta_N(s)$ has at most one zero in $[1-(4\log d_N)^{-1}, 1]$, we have this region as a zero free region for $L(s,\chi)$. Applying Lemma \ref{Stark-L(1,chi)} with $\sigma_1=2$, we obtain 
\begin{equation*}
    L(1,\chi) > c_1 \, \frac{1}{4\log d_N} \, \frac{1}{d_L^{-1/2}}\, c_2^{n_L}, 
\end{equation*}
where $c_1$ and $c_2$ are absolute constants. Hence, 
\begin{equation*}
    \frac{\log L(1,\chi)}{\log d_N} \gg \frac{\log d_L}{\log d_N} + \frac{\log\log d_N}{\log d_N} + \frac{n_L}{\log d_N}.
\end{equation*}
Since $\log d_{L_i}/\log d_{N_i} \to 0$ and $n_{N_i}/{\log d_{N_i}} \to 0$, we deduce that $\liminf \frac{\log L(1,\chi)}{\log d_{N_i}} =0$. Therefore,
\begin{equation*}
    \liminf_i \frac{\log \rho_{N_i}}{\log \sqrt{d_{N_i}}} = \liminf_i \frac{\log \rho_{L_i}}{\log \sqrt{d_{N_i}}}.
\end{equation*}
By assumption, (BS) holds for $\{L_i\}$, and hence $\lim_i \log \rho_{L_i}/\log d_{L_i} = 0$. Since $\log d_{N_i}\geq \log d_{L_i}$, we deduce from above that $\liminf_i \frac{\log \rho_{N_i}}{\log \sqrt{d_{N_i}}} =0$ and (BS) holds for $\{N_i\}$.



\end{proof}

\bigskip

\section{Asymptotically good families}
\bigskip

Let $L$ be an almost $S_n$-field with $Gal(\widetilde{L}/L)\cong S_m$. Let $M/L$ be an extension satisfying $M\cap \widetilde{L}=L$. Suppose $M/L$ is either almost normal or has solvable Galois closure. Let
\begin{equation*}
    \beta_0 := \max\bigg(1-\frac{1}{32 g_M}, 1 - \frac{c}{n^{e(n)} \delta(n) g_M} \bigg),
\end{equation*}
where the second term is as in \eqref{region}. Recall that $g_M:=\log \sqrt{d_M}$. By Theorem \ref{Stark-almost-normal} and Theorem \ref{Km}, we deduce that if $\zeta_M(s)$ has a zero $\beta$ in the region $[1-\beta_0, 1]$, then there exists a field $N$ such that $L\subseteq N\subseteq M$, $\zeta_N(\beta)=0$ and $[N:L]\leq 2$. If $[N:L]=2$, by Theorem \ref{zero-free-almost-S_n}, we have that $\zeta_N/\zeta_L(s)$ does not vanish on the interval $[1-\frac{1}{4^{m+2} \log d_M}, 1]$. Hence, $\zeta_M/\zeta_L(s)$ does not vanish in the region $[1-\beta_1, 1]$, where
$$
    \beta_1 := \max\left(\beta_0, 1 - \frac{1}{4^{m+2} \log d_M}\right).
$$


\begin{proof}[Proof of Theorem \ref{asymp-good-theorem}]
Write
\begin{equation*}
\zeta_M (s)= \frac{\rho_M}{(s-1)} F_M(s).
\end{equation*}
Taking $\log$ on both sides and dividing by $g_M$, we get for $s = 1+\theta_M$
\begin{equation*}\label{main}
\frac{\log \zeta_M (1+\theta_M)}{g_M} = \frac{\log \rho_M}{g_M} + \frac{\log F_M(1+ \theta_M)}{g_M} - \frac{\log \theta_M}{g_M}.
\end{equation*}
In \cite{TV}, it is shown that for any asymptotically exact family of number fields,
\begin{equation*}
\limsup_{i\to\infty} \frac{\log \rho_{M_i}}{g_{M_i}} \leq \sum_q \phi_q(\rM) \log \frac{q}{q-1}.
\end{equation*}
In order to prove the Theorem \ref{asymp-good-theorem}, it suffices to show that for the asymptotically good tower $\mathcal{M} = \{M_i\}$ of number fields, there is a choice of $\theta_{M_i} \to 0$, such that

\begin{equation}\label{zeta}
\liminf_{i\to \infty} \frac{\zeta_{M_i}(1+\theta_{M_i})}{g_{L_i}} \geq \sum_q \phi_q(\rM) \log \frac{q}{q-1},
\end{equation}
\begin{equation}\label{F(s)}
\limsup_{i\to\infty} \frac{\log F_{M_i}(1+ \theta_{M_i})}{g_{M_i}} \leq 0,
\end{equation}
and
\begin{equation}\label{theta}
\lim_{i\to\infty} \frac{\log \theta_{M_i}}{g_{M_i}} = 0.
\end{equation}
We first prove \eqref{F(s)} and make our choice of $\theta_{L_i}$'s. Write
\begin{equation*}
    Z_M(s) : = \frac{d}{ds} \log F_M(s).
\end{equation*}
Using Mellin transform of the Chebyshev step function, we have
\begin{equation}\label{L-O}
\frac{Z_M(s)}{s} =  \int_{1}^{\infty} (G_M(x) - x) \, x^{-s-1} \, dx - \frac{1}{s},
\end{equation}
for $\Re(s) > 1$, where
\begin{equation*}
G_M(x) := \sum_{\substack{q,m>1 \\ q^m \leq x}} N_q(M) \log q.
\end{equation*}
The unconditional Lagarias-Odlyzko \cite{Lag} estimate for $G_M(x)$ gives
\begin{equation*}
|G_M(x) - x| \leq C_4 \, x \exp \left( -C_5 \sqrt{\frac{\log x}{n}}\right) + \frac{x^\beta}{\beta}
\end{equation*}
for $\log x \geq C_6 \, n_M \, g_M^2$, where $C_4$, $C_5$, $C_6$ are positive absolute constants. Here, $\beta = \max(\beta', 1- (4\log d_M)^{-1})$, where $\beta'$ is the possible real exceptional zero of $\zeta_M(s)$ in $[1- (4\log d_M)^{-1},1 ]$.
For an asymptotically good family $\rM=\{M_i\}$, we can find positive constants $C_0$ and $C_{00}$ depending on $\rM$ such that
\begin{equation*}
C_0 \, n_{M_i} \leq g_{M_i }\leq C_{00} \, n_{M_i}.
\end{equation*}
First assume that $\zeta_M(s)$ does not vanish in the interval $[1-\beta_1, 1]$. Then,
\begin{equation*}
|G_M(x) - x| \leq C_4 \, x \exp \left( -C_5 \sqrt{\frac{\log x}{n_M}}\right) + \frac{x^{\beta_1}}{\beta_1}.
\end{equation*}
For all $x$ satisfying $\log x > \max( 4^{m+2} g_M,\,\, n^{e(n)} \delta(n) g_M)$, we have
\begin{equation*}
    \frac{x^{\beta_1}}{\beta_1} \ll x \exp \left(-C_5' \sqrt{\frac{\log x}{n_M}}\right).
\end{equation*}
Setting $s = 1 + \theta$ in \eqref{L-O}, we have
\begin{equation*}
\left| \frac{Z_M(1+\theta)}{(1+\theta)} \right| = \left|  \int_1^\infty (G_M(x) - x) \, x^{-2-\theta} \, dx \right| +O(1)
\end{equation*}
\begin{equation}\label{integral}
=\left| \int_1^{I} (G(x)-x) x^{-2-\theta} dx +  \int_{I}^{\infty} (G(x) - x) x^{-2-\theta} dx \right| + O(1).
\end{equation}
In case of the first integral, we use the following bound on $G_M(x)$,
\begin{equation*}
G_M(x) = \sum_{\substack{q,m\geq 1\\ q^m\leq x}} N_q(M) \log q \leq n_M \sum_{\substack{q,m \geq 1\\ 
q^m \leq x}} \log q \ll g_M \, x \log x.
\end{equation*}
Therefore, for some constant $C_{11}>0$, we have
\begin{equation*}
|G_M(x) - x| \leq C_{11} \, g_M \, x \log x.
\end{equation*}
Thus the integral
\begin{align*}
 \int_1^{I} (G(x)-x) x^{-2-\theta} dx & \leq C_{11} g_M \int_1^{I} x^{-1-\theta} \log x dx \\
    &\leq C_{11} g_M \, I \left(1 - \exp\left(-\theta \left(I\right) \right)\right) \ll I g_M.\\
\end{align*}

We now bound the second integral in \eqref{integral}. By the Lagarias-Odlyzko estimate \eqref{L-O}, we have
\begin{equation}\label{bound3}
\int_{I}^{\infty} (G(x) - x) x^{-2-\theta} dx \leq C_4 \int_{I}^{\infty} \exp \left(-C_5' \sqrt{\frac{\log x}{g_M}}\right) x^{-1-\theta} dx.
\end{equation}
Using change of variables
\begin{equation*}
x = y^{g_M \log y},
\end{equation*}
we get the right hand side of \eqref{bound3} as
\begin{equation}\label{bound4}
 2C_4 \int_{\exp \left( \sqrt{\frac{\log I}{g_M}} \right)}^\infty \,\, g_M \,
y^{-C_5' - 1 - \theta g_M \log y} \log y \, dy.
\end{equation}
For large $g_M$ and any fixed $\epsilon>0$, we bound $\log y \leq y^{\epsilon}$ to get \eqref{bound4} to be
\begin{equation}\label{bound}
 \leq 2C_4 \int_{\exp \left( \sqrt{\frac{\log I}{g_M}} \right)}^\infty \,\, g_M \, y^{-C_5' - 1+\epsilon - \theta g_M \log y} dy.
\end{equation}
We further know that in the above interval,
\begin{equation*}
\log y \geq \sqrt{\frac{\log I}{g_M}}.
\end{equation*}
Hence, \eqref{bound} is
\begin{equation}\label{bound2}
\leq 2C_4 \int_{\exp \left( \sqrt{\frac{\log I}{g_M}} \right)}^\infty g_M\, y^{-C_5' - 1+\epsilon - \theta \alpha(g_M)} dy,
\end{equation}
where 
\begin{equation*}
\alpha(g_M) = \sqrt{\frac{\log I}{g_M}}.
\end{equation*} 
Evaluating the integral \eqref{bound2}, we have
\begin{equation*}
\frac{2C_4\, g_M}{C_5' - \epsilon + \theta \alpha(g_M)} \exp \left( - (C_5' -\epsilon + \theta \alpha(g_M)) \sqrt{\frac{\log I}{g_M}} \right).
\end{equation*}
We now choose $I = \max\left(4^{m+2}\log d_M,\,\, n_M^{e(n_M)} \delta(n_M) \,g_M\right)$. Here $\delta(n_M)\ll g_M^4$ and $e(n_M)\ll \log n_M\ll \log g_M$. Choosing $\theta_M$ as
\begin{equation*}
\theta_M := (I g_M)^{-1},
\end{equation*}
we get
\begin{equation*}
\log F_M(1+\theta_M)= \int_{0}^{\theta_M} Z_M(1+\theta) d\theta \ll 1.
\end{equation*}
Therefore, \eqref{F(s)} holds. Furthermore, we also have
$$
    n_M^{e(n_M)} \delta(n_M) \,g_M \ll g_M^{C \log g_M},
$$
for some absolute constant $C$ and since $m\leq n-2$, we have $n_M \geq m (m+1)$ and hence
$$
     4^{m+2}\ll 4^m \ll 4^{\sqrt{n_M}} \ll 4^{C_5 \sqrt{g_M}}.
$$ 
Therefore, 
\begin{equation*}
\frac{\log \theta_{M_i}}{g_{M_i}} =  O\left(\frac{\sqrt{g_M}}{g_M}\right) + O\left( \frac{(\log g_M)^2}{g_M}\right)
\end{equation*}
and hence $\lim_i \log \theta_{M_i} / g_{M_i} =0$. Hence, we also get \eqref{theta}. For \eqref{zeta}, we use the fact that $\{M_i\}$ is a tower of number fields. Note that
\begin{equation*}
\frac{\zeta_{M_i}(1+\theta)}{g_{M_i}} = \sum_q \frac{N_q(M_i)}{g_{M_i}} \log \frac{1}{1-q^{-1-\theta}}
\end{equation*}
\begin{equation*}
= \sum_p  \frac{N_p(M_i)}{g_{M_i}} \log \frac{1}{1-p^{-1-\theta}} + \sum_{q=p^{k}, k>1} \frac{N_q(M_i)}{g_{M_i}} \log \frac{1}{1-q^{-1-\theta}}
\end{equation*}
In a tower, we know that $\phi_p(\rM) \leq \frac{N_p(M_i)}{g_{M_i}}$. Therefore,
\begin{equation*}
\sum_p  \frac{N_p(M_i)}{g_{M_i}} \log \frac{1}{1-p^{-1-\theta}} \geq \sum_p \phi_p  \log \frac{1}{1-p^{-1-\theta}},
\end{equation*}
for any $\theta>0$. We also have
\begin{equation*}
\sum_{q=p^{k}, k>1} \frac{N_q(M_i)}{g_{M_i}} \log \frac{1}{1-q^{-1-\theta}} \to \sum_{q=p^{k}, k>1} \phi_q \log \frac{1}{1-q^{-1-\theta}} 
\end{equation*}
uniformly for $\theta > -\delta$, for some $\delta>0$. Hence, we get
\begin{equation*}
\liminf_{i\to \infty} \zeta_{M_i}(1+\theta_{M_i}) \geq \sum_q \phi_q \log \frac{q}{q-1}.
\end{equation*}
Suppose $\zeta_M(s)$ admits a zero in the region $[1-\beta_1, 1]$. Then, we have that $\zeta_L(\beta_1)=0$ and the Artin $L$-function $\zeta_M/\zeta_L(s)$ is non-vanishing in $[1-\beta_1, 1]$. We can again apply Lagarias-Odlyzko's estimate as above to deduce $\log \left(\frac{F_M}{F_L}(1+\theta_M)\right) \ll 1$  and hence
\begin{equation*}
    \log F_M(1+\theta_M) \ll \log F_L(1+\theta_M).
\end{equation*}
By our assumption, GBS holds for $\{L_i\}$ and hence 
\begin{equation*}
    \frac{\log F_{L_i}(1+\theta_{M_i}) - \log \theta_{M_i}}{g_{L_i}} \to 0.
\end{equation*}
Since $g_{M_i}\geq g_{L_i}$, we deduce 
\begin{equation*}
    \frac{\log F_{M_i}(1+\theta_{M_i}) - \log \theta_{M_i}}{g_{M_i}} \to 0.
\end{equation*}
Therefore, GBS holds for $\rM=\{M_i\}$.
\end{proof}

\bigskip

\section{Brauer-Siegel for $S_n$ fields under Artin's conjecture}
\bigskip

In \cite[Theorem 4]{Stk}, Stark showed that Conjecture \ref{BS-conj} (BS) holds under the assumption of Artin's holomorphy conjecture, which asserts that Artin $L$-functions are analytic on $s\in\C\setminus\{1\}$.
In this section, we explicitly capture the Artin $L$-functions, whose holomorphy near $s=1$ shall imply (BS) for a family of $S_n$-fields.
\medskip

Let $K/\Q$ be an $S_n$-field.  Write
    $$
        \zeta_K(s) = L(s, \Ind^{S_n}_{S_{n-1}} 1) = \zeta(s) L(s,\rho_{(n-1)}),  
    $$
where $\rho_{(n-k)}$ is the representation of $S_n$ corresponding to the partition $[n-k,1,1,\ldots, 1]$.  Denote by $L\subset \widetilde{K}$, the fixed field of $S_{n-4}\times S_2 \times S_2$. Then,
    \begin{equation*}
        \zeta_L(s)   = L(s,\Ind^{S_n}_{S_{n-4} \times S_2 \times S_2} 1).
    \end{equation*}
 We compute the decomposition of this representation. Since $S_{n-4}\times S_2 \times S_2$ corresponds to the partition $[n-4,2,2]$, the representation $\Ind^{S_n}_{S_{n-4} \times S_2 \times S_2}$ is a Young permutation module, whose irreducible components are precisely the Specht modules $S^{\mu}$ such that $\mu$ dominates $[n-4,2,2]$. The multiplicity of each of this irreducible representation is given by the Kostka number, which can be calculated by the number of Young tableaux with the given content $n-4$ one's, $2$ two's and $2$ three's. Using this, we have the decomposition  of the form
    \begin{equation*}
        \Ind^{S_n}_{S_{n-4}\times S_2 \times S_2} 1 = 1 \oplus 2 \rho_{(n-1)} \oplus \rho_{(n-2)} \oplus \rho_{(n-3)}.
    \end{equation*}
    Therefore,
    \begin{equation*}
        \zeta_L(s) = \zeta(s) L^2(s,\rho_{(n-1)}) L(s,\rho_{(n-2)}) L(s, \rho_{(n-3)}).
    \end{equation*}
Any real zero $\beta$ of $\zeta_K(s)$ with $1/2\leq \beta \leq 1$ must be  zero of $L(s,\rho_{(n-1)})$, since $\zeta(s)$ does not vanish in this region. Suppose Artin's conjecture holds true for $L(s,\rho_{(n-2)})$ and $L(s,\rho_{(n-3)})$. Then, they are entire and do not have pole at $\beta$. Thus, $\beta$ is a double zero of $\zeta_L(s)$. By Stark's theorem, there is at most one zero of $\zeta_L(s)$ in the region $(1- (4\log d_L)^{-1}, 1)$, which is simple. Thus, we conclude that
    \begin{equation*}
        1-\beta > \frac{1}{4\log d_L}.
    \end{equation*}
But
\begin{equation*}
        d_L = d_K^n  N_{K/\Q} (d_{L/K}) \leq d_K^n \prod_{p|d_K} (N_{K/\Q}\,\,\, p)^{[L:K]} \leq d_K^{n+ n(n-1)(n-2)(n-3)/4}\leq d_K^{n^4/4}.
\end{equation*}
Thus
    \begin{equation}\label{BS-AHC}
        1-\beta > \frac{1}{4\log d_L} \geq \frac{1}{4 n^4 \log d_K}.
    \end{equation}
Hence, we deduce the following proposition.




\begin{proposition}
    Let $\rK=\{K_i\}$ be a family of $S_{n_i}$-fields satisfying $d_{K_i}^{1/n_{K_i}} \to \infty$. Under the assumption of Artin holomorphy conjecture for $L(s,\rho_{n_i-2})$ and $L(s,\rho_{n_i-3})$ for all $i$'s, the Conjecture \ref{BS-conj} holds for $\rK$.
\end{proposition}

\begin{proof}
     It suffices to show that 
    \begin{equation*}
        \liminf_i \frac{\log \rho_{K_i}}{\log d_{K_i}} \geq 0.
    \end{equation*}
    By Lemma \ref{Stark's lemma}
    \begin{equation*}
        \rho_{K_i} \geq c (1-\beta_i),
    \end{equation*}
    where $\beta_i = \max(\beta_i', (4\log d_{K_i})^{-1}) $ and $\beta_i'$ denotes the possible real zero of $\zeta_{K_i}(s)$ in the region $[1-(4\log d_K)^{-1},1)$. By \eqref{BS-AHC}, we get
    \begin{equation*}
        \log \rho_{K_i} \geq \log c + \log (4n_i^4 \log d_{K_i}) \gg \log (\log^4 d_{K_i}),
    \end{equation*}
    which is sufficient to establish Conjecture \ref{BS-conj}.
\end{proof}
\bigskip


\section*{Acknowledgements}
I thank Prof. V. Kumar Murty for several helpful discussions.


\begin{thebibliography}{00}
\bibitem{BS}
R. Brauer, On zeta-functions of algebraic number fields, {\it Amer. J. Math.}, \textbf{2}, pp. 243--250, (1947).

\bibitem{asif}
P. J. Cho, R. J. Lemke Oliver, A. Zaman, Effective Brauer-Siegel theorems for Artin $L$-functions, {\it arXiv:2510.02309}.

\bibitem{Dixit}
A. B. Dixit, On the generalized Brauer-Siegel theorem for asymptotically exact families with solvable Galois closure, {\it Int. Math. Res. Not. (IMRN)}, \textbf{2021}, no. 14, pp. 10941--10956, (2021).

\bibitem{Hoff}
J. Hoffstein and N. Jochnowitz, On Artin’s conjecture and the class number of certain CM fields, I, {\it Duke Math. J.}, \textbf{58}, no. 1, pp. 553--563, (1989).

\bibitem{Lag}
J. C. Lagarias and A. M. Odlyzko, Effective versions of the Chebotarev density theorem, in {\it Algebraic Number Fields}, Academic Press, ed. A. Fr\"ohlich, pp. 409--464, (1977).

\bibitem{lang}
S. Lang, Algebraic number theory, {\it Graduate Texts in Mathematics}, Springer-Verlag, Berlin--New York, (1994).

\bibitem{Km}
V. K. Murty, Class number of CM-fields with solvable normal closure, {\it Compos. Math.}, \textbf{127}, pp. 273--287, (2001).

\bibitem{Sie}
C. L. Siegel, \"Uber die Classenzahl quadratischer Zahlk\"orper, {\it Acta Arith.}, \textbf{1}, pp. 83--86, (1953).


\bibitem{Stk}
H. M. Stark, Some effective cases of the Brauer-Siegel theorem, {\it Invent. Math.}, \textbf{23}, pp. 135--152, (1974).

\bibitem{TV}
M. A. Tsfasman and S. G. Vl\u{a}du\c{t}, Asymptotic properties of global fields and generalized Brauer-Siegel theorem, {\it Mosc. Math. J.}, \textbf{2}, no. 2, pp. 329--402, (2002).

\bibitem{peng}
P. J. Wong, On Stark's class number conjecture and the generalised Brauer-Siegel conjecture, {\it Bull. Aust. Math. Soc.}, \textbf{106}, pp. 288--300, (2022).

\bibitem{Zyk}
A. Zykin, The Brauer-Siegel and Tsfasman-Vl\u{a}du\c{t} theorems for almost normal extensions of number fields, {\it Mosc. Math. J.}, \textbf{5}, pp. 261--268, (2005).















\end{thebibliography}
\end{document}